\documentclass{aptpub}
\usepackage{mathtools}
\def\P{{\mathbb P}}
\def\Z{{\mathbb Z}}
\def\E{{\mathbb E}}
\authornames{A. Maheshwari and P. Vellaisamy} 
\shorttitle{Long-range dependence of fractional processes} 

\begin{document}

\title{O\lowercase{n the} L\lowercase{ong-range} D\lowercase{ependence of} F\lowercase{ractional}\\ P\lowercase{oisson and} N\lowercase{egative} B\lowercase{inomial} P\lowercase{rocesses}} 

\authorone[Indian Institute of Technology Bombay]{A. Maheshwari} 
\authortwo[Indian Institute of Technology Bombay]{P. Vellaisamy}

\addressone{\small Department of Mathematics,
 Indian Institute of Technology Bombay, Powai, Mumbai 400076, INDIA. Email: aditya@math.iitb.ac.in} 
\addresstwo{\small Department of Mathematics,
 Indian Institute of Technology Bombay, Powai, Mumbai 400076, INDIA. Email: pv@math.iitb.ac.in} 
\begin{abstract}
We study the long-range dependence (LRD) of the increments of the fractional Poisson process (FPP), the fractional negative binomial process (FNBP) and the increments of the FNBP. We first point out an error in the proof of Theorem 1 of Biard and Saussereau \cite{biardapp} and prove that the increments of the FPP has indeed the short-range dependence (SRD) property, when the fractional index $\beta$ satisfies $0<\beta<\frac{1}{3}$. We also establish that the FNBP has the LRD property, while the increments of the FNBP possesses the SRD property. 

\end{abstract}



\section{Introduction}
\noindent The long-range dependence (LRD) property of a stochastic model or a process has been widely studied in the literature. It has applications to several areas such 
as Internet data traffic modeling \cite{karag2004}, finance \cite{Ding1993}, econometrics \cite{Pagan1996}, hydrology \cite[p. 461-472]{DouTaqqu2003}, climate studies \cite{climate2006} and {\it etc}.  
Let $\{N_{\beta}(t) \}_{t \geq 0}$ be a fractional Poisson process (see \cite{lask}), where we call $\beta$ the fractional index. The LRD property of the fractional Poisson process (FPP) is proved in \cite{LRD2014}. Recently, Biard and Saussereau \cite{biardapp} showed that the fractional Poissonian noise (FPN) process $\{Z^{1}_{\beta}(n-1)\}_{n\geq 1}$, defined by
\begin{equation} \label{z-beta}
Z_{\beta}^{\delta}(t)=N_{\beta}(t+\delta,\lambda)-N_{\beta}(t,\lambda),~0<\beta<1,~\delta>0,~t\geq0,
\end{equation}
 has the LRD property. Note that there are several definitions for the LRD property of a stochastic process. The proof of Biard and Saussereau \cite{biardapp} uses the definition given in \cite{heyde97} which  is based on showing that $\lim_{m\rightarrow\infty}\Delta_{n}^{(m)}$ (see \eqref{def-lrd-eq}) is infinite. Unfortunately, there is a mistake in the proof. We  show that $\lim_{m\rightarrow\infty}\Delta_{n}^{(m)}$ is finite (see Theorem \ref{theorem-finite}), which disproves their claim that the
FPN has the LRD property. Using an alternate definition, we show that the FPN $\{Z_{\beta}^{\delta}(t)\}_{t\geq0}$ has the short-range dependence (SRD) property, when $\beta\in(0,\frac{1}{3})$.\\
  Let $\{Y(t)\}_{t \geq 0}$ be a gamma subordinator so that    $\{Q(t,\lambda)\}_{t\geq0}=\{N(Y(t),\lambda)\}_{t\geq0}$ is a  negative binomial process.
Very recently,
the fractional negative binomial process (FNBP) defined by $\{Q_{\beta}(t,\lambda)\}_{t\geq0} = \{N_{\beta}(Y(t),\lambda)\}_{t\geq0}$ is  studied in detail in \cite{fnbpfp}. We here prove that the FNBP has the LRD property. Let $\delta>0$ be fixed and define the increments of the FNBP as
 \begin{equation*}
  Q_{\beta}^{\delta}(t)=Q_{\beta}(t+\delta,\lambda)-Q_{\beta}(t,\lambda),~~~t\geq0,
 \end{equation*}
which we call  the fractional negative binomial noise (FNBN). We prove also that the FNBN has the SRD property.\\
The paper is organized as follows. In Section \ref{section:prelims}, some preliminary notations, results and definitions are stated. In Section \ref{section:FPN}, we discuss the proof of Theorem 1 of \cite{biardapp} and point out an error in their proof showing that FPN has the LRD property. We indeed prove that the FPN has the SRD property for the case $\beta\in(0,\frac{1}{3})$. In Section \ref{section:LRD-FNBP}, the LRD property of the FNBP and the SRD property of the FNBN are proved.

\section{Preliminaries}\label{section:prelims}
\noindent In this section, we introduce the notations and the results that will be used later. Let $\mathbb{Z}_{+}=\{0,1,\ldots\}$ be the set of nonnegative integers.
Let $\{N(t,\lambda)\}_{t\geq0}$ be a Poisson process with rate $\lambda>0$, so that
\begin{equation*}
p(n|t,\lambda)=\mathbb P[N(t,\lambda)=n]=\frac{(\lambda t)^{n}e^{-\lambda t}}{n!},~~~~n\in\Z_{+}. 
\end{equation*}
For $\alpha>0,~p>0$, let $\{Y(t)\}_{t\geq0}$ be a gamma subordinator, where $Y(t)\sim G(\alpha,pt)$ with density
\begin{equation}\label{gammaden}
g(y|\alpha,pt)=\frac{\alpha^{pt}}{\Gamma{(pt)}}y^{pt-1}e^{-\alpha y},~~~~y>0.	 
\end{equation}
We say a random variable $X$ follows a negative binomial distribution with parameters $\alpha>0$ and $0<\eta<1$, denoted by $\text{NB}(\alpha,\eta)$, if
\begin{equation}\label{nbpmf}
 \P[X=n]=\binom{n+\alpha-1}{n}\eta^{n}(1-\eta)^{\alpha},\,\,\,\,\,n\in\Z_{+}.
\end{equation}
When $\alpha$ is a natural number, then $X$ denotes the number of successes before the $\alpha$-th failure in a sequence of Bernoulli trials with success probability $\eta$.

\vspace{.2cm}
\subsection{Fractional Poisson and NB processes}
\begin{definition}
\noindent Let $0<\beta\leq1$. The fractional Poisson process (FPP) $\{N_{\beta}(t,\lambda)\}_{t\geq0}$, which is a generalization of the Poisson process $\{N(t,\lambda)\}_{t\geq0}$, is defined as the stochastic process whose $p_{_\beta}(n|t,\lambda)=\mathbb{P}[N_{\beta}(t,\lambda)=n]$ satisfies (see \cite{lask,main,mnv})
 \begin{flalign}
&&D^{\beta}_{t}p_{_{\beta}}(n|t,\lambda) &= -\lambda p_{_{\beta}}(n|t,\lambda)+\lambda p_{_{\beta}}(n-1|t,\lambda),\,\,\,\text{for } n\geq1,& \nonumber\\
 &&D^{\beta}_{t}p_{_{\beta}}(0|t,\lambda) &= -\lambda p_{_{\beta}}(0|t,\lambda)\nonumber,&
 \end{flalign}
  $\text{where } p_{_{\beta}}(n|0,\lambda)=1,\text{ if }n=0, \text{ and is zero if }n\geq1.$ Here, $D^{\beta}_{t}$ denotes the Caputo fractional derivative defined as
 \begin{equation*}
  D_{t}^{\beta}f(t)= \begin{cases} 
     \hfill \dfrac{1}{\Gamma(1-\beta)}\displaystyle\int\limits_{0}^{t}\frac{f'(s)}{(t-s)^{\beta}}ds, \hfill    &0<\beta<1 , \\ 
      f'(t), \,\,\,\,\,\,\,\,\,\,\, \beta=1,&
  \end{cases}
 \end{equation*}
where $f'$ denotes the derivative of $f.$ 
\end{definition}
\noindent The mean and the variance of the FPP are given by (see \cite{lask})
 \begin{align}
 \E [N_{\beta}(t,\lambda)] &= qt^{\beta},\label{fppmean} \\
 \mbox{Var}[N_{\beta}(t,\lambda)] &=qt^{\beta}\left[1+qt^{\beta}\left(\frac{\beta B(\beta, 1/2)}{2^{2\beta-1}}-1\right)\right],\nonumber
 \end{align}
 where $q=\lambda/\Gamma(1+\beta)$ and $B(a,b)$ denotes the beta function. An alternative form for Var[$N_{\beta}(t,\lambda)$] is given in \cite[eq. (2.8)]{BegOrs09} as
 \begin{equation}\label{alternative-fppvar}
  \mbox{Var}[N_{\beta}(t,\lambda)]=q t^{\beta}+\frac{(\lambda t^{\beta})^{2}}{\beta}\left(\frac{1}{\Gamma(2\beta)}-\frac{1}{\beta\Gamma^{2}(\beta)}\right).
 \end{equation}
 
\noindent It is also known that (see \cite{mnv}) when $0<\beta<1,$
 \begin{equation*}
 N_{\beta}(t,\lambda)\stackrel{d}{=}N(E_{\beta}(t),\lambda), 
 \end{equation*}
 where $\{E_{\beta}(t)\}_{t\geq0}$ is the inverse $\beta$-stable subordinator and is independent of $\{N(t,\lambda)\}_{t\geq0}$.\\
 Let $\{N(t,\lambda)\}_{t\geq0}$ be a Poisson process and $\{Y(t)\}_{t\geq0}$ be an independent gamma subordinator (see \eqref{gammaden}).  The negative binomial process $\{Q(t,\lambda)\}_{t\geq0}=\{N(Y(t),\lambda)\}_{t\geq0}$ is a subordinated Poisson process (see \cite{fell,kozubo}) with 
 \begin{equation*}
  \mathbb P[Q(t,\lambda)=n]=\delta(n|\alpha,pt,\lambda)=\binom{n+pt-1}{n}\eta^{n}(1-\eta)^{pt},
 \end{equation*}
 where $\eta=\lambda/(\alpha+\lambda)$. That is, $Q(t,\lambda)\sim \text{NB}(pt,\eta)$ for $t>0$, defined in \eqref{nbpmf}.
 \begin{definition}
 The fractional negative binomial process (FNBP) is defined as 
 \begin{equation*}
 Q_{\beta}(t,\lambda)=N_{\beta}(Y(t),\lambda), ~~~t\geq0, 
 \end{equation*}
 where $\{N_{\beta}(t,\lambda)\}_{t\geq0}$ is an FPP and is independent of $\{Y(t)\}_{t\geq0}$. 
 \end{definition}
  For more details and additional properties of the FNBP, the reader is referred to \cite{fnbpfp}.

\subsection{The LRD and the SRD property}
\noindent There are several definitions for the LRD and the SRD property of a stochastic process. We here present those definitions which will be used in this paper.\\
 The following definition is due to \cite{heyde97} and modified for non-centered process in \cite{biardapp}. Let $\{X_{m}\}_{m\geq1}$ be a discrete-time stochastic process. Define $S_{n}=\sum_{j=1}^{n}X_{j}$ and $\sigma_{n}^{2}=\sum_{j=1}^{n}\text{Var}[X_{j}]$, $n\geq1$. 
\begin{definition}\label{def-lrd}
 Let $\{X_{m}\}_{m\geq1}$ be a second order process (not necessarily stationary) with the block mean process 
 \begin{equation*}
  Y^{(m)}_{n}=\frac{S_{nm}-S_{(n-1)m}}{\sigma^{2}_{nm}-\sigma_{(n-1)m}^{2}},~n\geq1,
 \end{equation*}
 and $\Delta^{(m)}_{n}=\left(\sigma^{2}_{nm}-\sigma_{(n-1)m}^{2}\right)\left(\text{Var}[Y^{(m)}_{n}]\right)$.  The process $\{X_{m}\}_{m\geq1}$ has the LRD property if 
 \begin{equation*}
 \lim_{m\rightarrow\infty} \Delta^{(m)}_{n}=\infty.
 \end{equation*}

 \end{definition}
 \noindent We next present an alternate definition of the LRD and the SRD property (see \cite{ovi-lrd}).
 \begin{definition}\label{LRD-definition}
 Let $s>0$ be fixed and $t>s$. Suppose a stochastic process $\{X(t)\}_{t\geq0}$ has the correlation function Corr$[X(s),X(t)]$ that satisfies
 \begin{equation}\label{LRD-defn1}
c_1(s)t^{-d}\leq\text{Corr}[X(s),X(t)]\leq c_2(s)t^{-d},
 \end{equation}
 for large $t$, $d>0$, $c_1(s)>0$ and $c_2(s)>0$. In other words, 
 \begin{equation}\label{LRD-defn2}
\lim\limits_{t\to\infty}\frac{\text{Corr}[X(s),X(t)]}{t^{-d}}=c(s),
 \end{equation}for some $c(s)>0$ and $d>0.$  We say $\{X(t)\}_{t\geq0}$ has the LRD property if $d\in(0,1)$  and has the SRD property if $d\in(1,2)$.
 \end{definition}
 Note that \eqref{LRD-defn1} and \eqref{LRD-defn2} are equivalent and imply that Corr$[X(s),X(t)]$ behaves like $t^{-d}$, for large $t$.

\vspace{-.3cm}
\section{Dependence structure for the fractional Poisson process}\label{section:FPN}

\noindent First we require the following result (see \cite[Lemma 2]{biardapp}) which gives the factorial moments of the increments of the FPP. For simplicity, the parameter $\lambda$ is suppressed in $\{N_{\beta}(t,\lambda)\}_{t\geq0}$ and $\{Q_{\beta}(t,\lambda)\}_{t\geq0}$, when no confusion arises.
\begin{lemma}\label{factorial-moments-fpp-lemma}
 Let $0\leq s\leq t$ and $q= \lambda/ \Gamma(1+ \beta).$ Then
 \begin{align}\label{factorial-moments-fpp-eq}
  \mathbb{E}\big[(N_{\beta}(t)-N_{\beta}(s)) (N_{\beta}(t)-N_{\beta}(s)-1)\big]=2\beta q^{2}\int_{s}^{t}(t-r)^{\beta}r^{\beta-1}dr.
 \end{align}
\end{lemma}

\noindent  Note that for the FPN $\{Z^1_{\beta}(n-1)\}_{n\geq1}$, where $Z^1_{\beta}(n-1)=
N_{\beta}(n)-N_{\beta}(n-1)$ (see \eqref{z-beta}), 
 \begin{equation}\label{def-lrd-eq}
\Delta^{(m)}_{n}=\frac{\text{Var}[N_{\beta}(nm)-N_{\beta}((n-1)m)]}{\sum_{j=(n-1)m+1}^{nm}\text{Var}[N_{\beta}(j)-N_{\beta}(j-1)]}.
 \end{equation}
\noindent Biard and Saussereau \cite{biardapp} computed $\Delta_{n}^{(m)}$ for the FPN $\{Z^1_{\beta}(n-1)\}_{n\geq1}$ and showed that $\lim_{m\rightarrow\infty}\Delta_{n}^{(m)}$ is infinite.  We next show that $\lim_{m\rightarrow\infty}\Delta_{n}^{(m)}$ is indeed finite.
It is convenient to use the notation $C(x,y)=x^y-(x-1)^y$.
\begin{theorem}\label{theorem-finite}
Let $0<\beta<1$ and $\{Z^1_{\beta}(n-1)\}_{n\geq1}$ be the FPN. Then $\lim_{m\rightarrow\infty}\Delta_{n}^{(m)}$ is finite.
 \begin{proof}
  Our proof starts with the observation that, for $0\leq s\leq t$,
  \begin{align}
   \int_{s}^{t}(t-r)^{\beta}r^{\beta-1}dr&\leq\int_{0}^{t}(t-r)^{\beta}r^{\beta-1}dr \nonumber \\
	&=t^{2\beta}\int_{0}^{1}(1-u)^{\beta+1-1}u^{\beta-1}du~~~~(\text{substituting }r=tu)\nonumber\\
   &=t^{2\beta}B(\beta,1+\beta).\label{integral-ubound}
  \end{align}
Note  $\beta-1<0$ and $0\leq r\leq t$ implies that $ t^{\beta-1}\leq r^{\beta-1}$. Therefore, we have 
\begin{align}
 \int_{s}^{t}(t-r)^{\beta}r^{\beta-1}dr&\geq\int_{s}^{t}(t-r)^{\beta}t^{\beta-1}dr \nonumber\\
  &=\frac{t^{\beta-1}(t-s)^{\beta+1}}{\beta+1}.\label{integral-lbound}
\end{align}
Substituting \eqref{integral-ubound} and \eqref{integral-lbound} into \eqref{factorial-moments-fpp-eq} yields
\begin{equation}\label{integral-bound}
 c t^{\beta-1}(t-s)^{\beta+1}\leq \mathbb{E}\big[(N_{\beta}(t)-N_{\beta}(s)) (N_{\beta}(t)-N_{\beta}(s)-1)\big] \leq  2dt^{2\beta},
\end{equation}
where $c=2\beta q^{2}/(\beta+1)$ and $d=\beta q^2 B(\beta,1+\beta)$. Consider now
\begin{align}
 \mathbb{E}\big[(N_{\beta}(nm)-N_{\beta}((n-1)m))^{2}\big]&=\mathbb{E}[(N_{\beta}(nm)-N_{\beta}((n-1)m))(N_{\beta}(nm)-N_{\beta}((n-1)m)-1)]\nonumber\\
 &~~~~~~~~+\mathbb{E}[(N_{\beta}(nm)-N_{\beta}((n-1)m))]\nonumber\\
 &\leq 2d(nm)^{2\beta}+q((nm)^{\beta}-((n-1)m)^{\beta})~(\text{using }\eqref{integral-bound} \text{ and } \eqref{fppmean})\nonumber\\
  &= \big(2dn^{2\beta}+qC(n,\beta)m^{-\beta}\big)m^{2\beta}.\label{num-1}
\end{align}
Using \eqref{num-1}, we have
\begin{align}
 \text{Var}\big[N_{\beta}(nm)-N_{\beta}((n-1)m)\big]&=\mathbb{E}[(N_{\beta}(nm)-N_{\beta}((n-1)m))^{2}]\nonumber\\
 &~~~~~~~-\left(\mathbb{E}[N_{\beta}(nm)-N_{\beta}((n-1)m)]\right)^{2}\nonumber\\
 &\leq \big(2dn^{2\beta}+qC(n,\beta)m^{-\beta}\big)m^{2\beta}\nonumber\\
 &~~~~~~~-q^{2}C^{2}(n,\beta)m^{2\beta}\nonumber~~(\text{using }\eqref{fppmean})\\
 &=\big(2dn^{2\beta}+qC(n,\beta)m^{-\beta}-q^{2}C^{2}(n,\beta)\big)m^{2\beta}\label{numerator}.
\end{align}
Similarly,  we  see that for $j\geq1$
\begin{align}
 \mathbb{E}[(N_{\beta}(j)-N_{\beta}(j-1))^{2}]&=\mathbb{E}\big[\big(N_{\beta}(j)-N_{\beta}(j-1)\big)\big(N_{\beta}(j)-N_{\beta}(j-1)-1\big)\big]\nonumber\\
  &~~~~~~~~+\mathbb{E}\big[(N_{\beta}(j)-N_{\beta}(j-1))\big]\nonumber\\
 &\geq cj^{\beta-1}(j-(j-1))^{\beta+1}+qC(j,\beta)~~(\text{using }\eqref{integral-bound}\text{ and }\eqref{fppmean})\nonumber\\
 &=cj^{\beta-1}+qC(j,\beta),\nonumber
\end{align}
which leads to
\begin{align}
 \text{Var}\big[N_{\beta}(j)-N_{\beta}(j-1)\big]&=\mathbb{E}\big[(N_{\beta}(j)-N_{\beta}(j-1))^{2}\big]-\big(\mathbb{E}[N_{\beta}(j)-N_{\beta}(j-1)]\big)^{2}\nonumber\\
 &\geq cj^{\beta-1}+qC(j,\beta)-q^{2}C^{2}(j,\beta)\nonumber\\
 &\geq cj^{\beta-1}+qC(j,\beta)-q^{2}C(j,2\beta),\label{denom-3}
\end{align}
since $(a-b)^{2}\leq a^{2}-b^{2}$, for $a\geq b\geq0$. Using \eqref{denom-3}, we have
\begin{align}
 \sum_{j=(n-1)m+1}^{nm}\text{Var}[N_{\beta}(j)-N_{\beta}(j-1)]&\geq \sum_{j=(n-1)m+1}^{nm}\Big(cj^{\beta-1}+qC(j,\beta)-q^{2}C(j,2\beta)\Big)\nonumber\\
 &\geq cm(nm)^{\beta-1}+q((nm)^{\beta}-((n-1)m)^{\beta})\nonumber\\
 &~~~~~~-q^{2}((nm)^{2\beta}-((n-1)m)^{2\beta}),\nonumber
 \shortintertext{using  $\sum_{j=l+1}^{k} C(j, \beta)= k^{\beta}- l^{\beta}$. Therefore, we get}
\sum_{j=(n-1)m+1}^{nm}\text{Var}[N_{\beta}(j)-N_{\beta}(j-1)] &\geq m^{2\beta}\left[\Big(cn^{\beta-1}\right.+\left.qC(n,\beta)\Big)m^{-\beta}-q^{2}C(n,2\beta)\right].\label{denominator}
\end{align}
From \eqref{numerator} and \eqref{denominator}, we conclude that
\begin{align}
 \Delta^{(m)}_{n}&=\frac{\text{Var}[N_{\beta}(nm)-N_{\beta}((n-1)m)]}{\sum_{j=(n-1)m+1}^{nm}\text{Var}[N_{\beta}(j)-N_{\beta}(j-1)]}\nonumber\\
 &\leq\frac{\big[2dn^{2\beta}+qC(n,\beta)m^{-\beta}-q^{2}C^{2}(n,\beta)\big]m^{2\beta}}{\big[\big(cn^{\beta-1}+qC(n,\beta)\big)m^{-\beta}-q^{2}C(n,2\beta)\big]m^{2\beta}}\nonumber\\
 &\stackrel{m\rightarrow\infty}{\longrightarrow}\frac{2dn^{2\beta}-q^{2}C^{2}(n,\beta)}{-q^{2}C(n,2\beta)}~~~~(\text{since}~m^{-\beta}\rightarrow 0,~\text{as }m\rightarrow\infty)\nonumber\\
 &=\frac{C^{2}(n,\beta)-(2dn^{2\beta}/q^{2})}{C(n,2\beta)}\leq \frac{C^{2}(n,\beta)}{C(n,2\beta)}\leq 1.\nonumber
 \end{align}
Since $\Delta^{(m)}_{n}\geq0$, we see that $\lim_{m\rightarrow\infty}\Delta^{(m)}_{n}\in[0,1]$  and hence the result follows.
\end{proof}
\end{theorem}
\begin{remark}
For $t\in \mathbb{Z}_+\backslash\{0\}$ and $0<h<1$, Biard and Saussereau (see \cite[Theorem 1]{biardapp}) showed that $\lim_{m\rightarrow\infty}\Delta^{(m)}_{t}$ is infinite using the following inequality in the proof of Theorem 1
\begin{equation*}
\int_{tm-m}^{tm}(tm-r)^{h}r^{h-1}dr=(tm)^{2h}\int_{1-1/t}^{1}(1-u)^{h}u^{h-1}du\geq (tm)^{2h}B(1+h,h),
\end{equation*}
which is unfortunately incorrect, since
\begin{equation*}
\int_{1-1/t}^{1}(1-u)^{h}u^{h-1}du\leq \int_{0}^{1}(1-u)^{h}u^{h-1}du=B(1+h,h).
\end{equation*}
\end{remark}

 The remainder of this section is devoted to the proof of the SRD property of the FPN $\{Z^{\delta}_{\beta}(t)\}_{t\geq0}$, for $0<\beta<\frac{1}{3}$. 
\begin{definition}
	Let $f(x)$ and $g(x)$ be positive functions. We say that $f(x)$ is asymptotically equal to $g(x)$, written as $f(x)\sim g(x)$, as $x$ tends to infinity, if 
	\begin{equation*}
	\lim\limits_{x\rightarrow\infty}\frac{f(x)}{g(x)}=1.
	\end{equation*}
\end{definition}

\begin{theorem} Let $0<\beta <\frac{1}{3}.$ Then
 the FPN $\{Z^{\delta}_{\beta}(t)\}_{t\geq0}$ has the SRD property.
 \end{theorem}
 \begin{proof}

 Let $s,\delta \geq 0$ be fixed and $0\leq s+\delta\leq t$. We start with
 \begin{align}
 \text{Cov}[Z^{\delta}_{\beta}(s),Z^{\delta}_{\beta}(t)]&= \text{Cov}\left[N_{\beta}(s+\delta)-N_{\beta}(s),N_{\beta}(t+\delta)-N_{\beta}(t)\right]\nonumber\\
 &=\text{Cov}\left[N_{\beta}(s+\delta),N_{\beta}(t+\delta)\right]+\text{Cov}\left[N_{\beta}(s),N_{\beta}(t)\right]\nonumber\\
 &~~~~~~~~~~~-\text{Cov}\left[N_{\beta}(s+\delta),N_{\beta}(t)\right]-\text{Cov}\left[N_{\beta}(s),N_{\beta}(t+\delta)\right].\label{covariance-xbeta}
 \end{align}
\noindent It is known that (see \cite[p. 9]{LRD2014})
 \begin{equation}\label{lrd-fpp}
\text{Cov}\left[N_{\beta}(s),N_{\beta}(t)\right]=qs^{\beta}+q^{2}\left[\beta s^{2\beta}B(\beta,1+\beta)+F(\beta;s,t)\right],
 \end{equation}
 where $F(\beta;s,t)=\beta t^{2\beta}B(\beta,1+\beta;s/t)-(st)^{\beta}$ and $B(a,b;x)=\int_{0}^{x}u^{a-1}(1-u)^{b-1}du~\text{for}~ a>0,~b>0,$ is the incomplete beta function. The asymptotic expansion of $F(\beta;s,t)$ for fixed $s$ and for large $t$ (see \cite{LRD2014}) is given by
 \begin{align}
 F(\beta;s,t)&=\frac{-\beta}{\beta+1}\left(s/t\right)^{\beta+1}+O\left(\left(s/t\right)^{\beta+2}\right)\nonumber\\
 &\sim \frac{-\beta}{\beta+1}\left(s/t\right)^{\beta+1}.\label{fbetaasym}
 \end{align}
 Combining \eqref{lrd-fpp} with \eqref{covariance-xbeta}, we deduce that
 \begin{align}
 \text{Cov}[Z^{\delta}_{\beta}(s),Z^{\delta}_{\beta}(t)]&=q^{2}\left[F(\beta;s+\delta,t+\delta)+F(\beta;s,t)-F(\beta;s+\delta,t)-F(\beta;s,t+\delta)\right]\nonumber\\
 &\sim\frac{-\beta q^{2}}{\beta+1}\left[\left(\tfrac{s+\delta}{t+\delta}\right)^{\beta+1}+\left(\tfrac{s}{t}\right)^{\beta+1}-\left(\tfrac{s+\delta}{t}\right)^{\beta+1}-\left(\tfrac{s}{t+\delta}\right)^{\beta+1}\right]~(\text{using } \eqref{fbetaasym})\nonumber\\
 & =\frac{-\beta q^{2}}{\beta+1}\left((s+\delta)^{\beta+1}-s^{\beta+1}\right)\left((t+\delta)^{-(\beta+1)}-t^{-(\beta+1)}\right)\nonumber\\
&=\frac{-\beta q^{2}}{\beta+1}\left((s+\delta)^{\beta+1}-s^{\beta+1}\right)t^{-(\beta+1)}\left((1+\delta/t)^{-(\beta+1)}-1\right)\nonumber\\
&\sim \beta q^{2}\delta\left((s+\delta)^{\beta+1}-s^{\beta+1}\right)t^{-(\beta+2)}~(\text{using binomial expansion})\nonumber\\
&=Kt^{-(\beta+2)},\label{covapproxA}
 \end{align}
  where $K=\beta q^{2}\delta((s+\delta)^{\beta+1}-s^{\beta+1})$. Observe that
 \begin{align}
 \text{Var}[Z^{\delta}_{\beta}(t)]
 &=\text{Var}[N_{\beta}(t+\delta)]+\text{Var}[N_{\beta}(t)]-2\text{Cov}[N_{\beta}(t),N_{\beta}(t+\delta)].\nonumber
 \end{align}
Denote $R=\frac{\lambda^2}{\beta}\left(\tfrac{1}{\Gamma(2\beta)}-\tfrac{1}{\beta\Gamma^2(\beta)}\right)=2d-q^2$. Using \eqref{alternative-fppvar} and \eqref{lrd-fpp}, we get
 \begin{align}
 \text{Var}[Z^{\delta}_{\beta}(t)]&=q[(t+\delta)^{\beta}+t^{\beta}]+R[(t+\delta)^{2\beta}+t^{2\beta}]\nonumber\\
 &~~~~~~~-2\big[qt^{\beta}+q^{2}(\beta B(\beta,1+\beta)t^{2\beta}+F(\beta;t,t+\delta))\big]\nonumber\\
& = q[(t+\delta)^{\beta}-t^{\beta}]+R[(t+\delta)^{2\beta}+t^{2\beta}]-2dt^{2\beta}-2q^{2}F(\beta;t,t+\delta),\nonumber\\
&= q[(t+\delta)^{\beta}-t^{\beta}]+R[(t+\delta)^{2\beta}+t^{2\beta}]-2dt^{2\beta}\nonumber\\
&~~~~~~~-2q^{2}\beta (t+\delta)^{2\beta} B(\beta,1+\beta;t/(t+\delta))+2q^{2}(t(t+\delta))^{\beta}.\nonumber
\end{align}
Since $ B(\beta,1+\beta;t/(t+\delta))\sim B(\beta,1+\beta)$ for large $t$, we have
\begin{align}
 \text{Var}[Z^{\delta}_{\beta}(t)]&= q[(t+\delta)^{\beta}-t^{\beta}]+(R-2d)[(t+\delta)^{2\beta}+t^{2\beta}]+2q^{2}(t(t+\delta))^{\beta}\nonumber\\
&= q[(t+\delta)^{\beta}-t^{\beta}]-q^{2}[(t+\delta)^{2\beta}+t^{2\beta}]+2q^{2}(t(t+\delta))^{\beta}\nonumber\\
&= q[(t+\delta)^{\beta}-t^{\beta}]-q^{2}\big[(t+\delta)^{\beta}-t^{\beta}\big]^{2}\nonumber\\
&= qt^{\beta}[(1+\delta/t)^{\beta}-1]-q^{2}t^{2\beta}\big[(1+\delta/t)^{\beta}-1\big]^{2}
\nonumber\\
&\sim \beta \delta qt^{\beta-1}-(\beta\delta q)^{2}t^{2(\beta-1)}~~(\text{using binomial expansion})\nonumber\\
&\sim \beta \delta qt^{\beta-1}\label{varC}.
 \end{align}
 Using \eqref{covapproxA} and \eqref{varC}, we finally obtain the correlation function, for large $t$,
 \begin{align}
 \text{Corr}[Z^{\delta}_{\beta}(s),Z^{\delta}_{\beta}(t)]&=\frac{\text{Cov}[Z^{\delta}_{\beta}(s),Z^{\delta}_{\beta}(t)]}{\sqrt{\text{Var}[Z^{\delta}_{\beta}(s)]}\sqrt{\text{Var}[Z^{\delta}_{\beta}(t)]}}\nonumber\\
 &\sim t^{-\frac{3}{2}(\beta+1)}S,\nonumber
 \end{align}
 where $S=\frac{K}{\sqrt{\beta\delta q}\sqrt{\text{Var}[Z^{\delta}_{\beta}(s)]}}.$ Thus, the correlation function of the FPN process decays at the rate $t^{-3(\beta+1)/2}$.  Since $3(\beta+1)/2\in(1.5,2)$, for $0<\beta<\frac{1}{3}$,
the result follows.
 \end{proof}
\section{Dependence structure for fractional negative binomial process}\label{section:LRD-FNBP}
\noindent In this section, we investigate the LRD property of the FNBP $\{Q_{\beta}(t)\}_{t\geq0}$, studied in detail in \cite{fnbpfp}, and the SRD property of the FNBN $\{Q^\delta_{\beta}(t)\}_{t\geq0}$.
For that purpose, we first need  the following result from \cite{fnbpfp} regarding the mean, variance and autocovariance functions of the FNBP. 
 \begin{theorem}\label{Theorem-fnbp-dist-properties} The mean, variance and autocovariance functions of the FNBP $\{Q_{\beta}(t)\}_{t\geq0}$ are given by\\
   (i)$~~~\E[Q_{\beta}(t)] = q\dfrac{\Gamma(pt+\beta)}{\alpha^{\beta}\Gamma(pt)}=q\E[Y^{\beta}(t)]\sim q\left(\dfrac{pt}{\alpha}\right)^{\beta}= \left(\dfrac{p}{\alpha}\right)^{\beta}\mathbb{E}[N_{\beta}(t)]$, for large $t$,\\
   (ii)$~~\text{Var}[Q_{\beta}(t)]=q\E[Y^{\beta}(t)]\left(1-q\E[Y^{\beta}(t)]\right)+2d\E[Y^{2\beta}(t)]$,
   \begin{flalign}
 \text{(iii)}~~\text{Cov}[Q_{\beta}(s),Q_{\beta}(t)]&=q\E[Y^{\beta}(s)]+ d\E[Y^{2\beta}(s)] -q^{2}\E[Y^{\beta}(s)]\E[Y^{\beta}(t)]&&\nonumber \\&~~~~~
  +q^{2}\beta \mathbb{E}[Y^{2\beta}(t)B(\beta,1+\beta;Y(s)/Y(t))].\nonumber&
 \end{flalign}
 \end{theorem}
\noindent We need the following result also.
\begin{lemma}
Let $0<\beta<1$ and $a\geq b\geq 0$. Then,
\begin{equation}\label{appnedix-1}
(a-b)^\beta\geq a^\beta-b^\beta.
\end{equation}
\begin{proof}
Since $b\geq0$, we have
\begin{align*}
&~~~~~ x-b\leq x,~x\in(b,a)\\
&\Rightarrow(x-b)^{\beta-1}\geq x^{\beta-1} ~(\text{since }\beta <1)\\
&\Rightarrow\int_{b}^{a}(x-b)^{\beta-1}dx\geq \int_{b}^{a}x^{\beta-1}dx\\
&\Rightarrow (a-b)^{\beta}\geq a^{\beta}-b^{\beta},
\end{align*}
since $0<\beta<1.$ 
\end{proof}
\end{lemma}  
The following is the key result used in the proof of the main result.
 \begin{lemma}\label{Lemma-asym}
Let $0<\beta<1$ and  $0< s< t $, where $s$ is fixed. Then, as $t$ tends to infinity,
(i) The asymptotic expansion of $\E[Y^{\beta}(s)Y^{\beta}(t)]$   is 
 \begin{align}
 \E\left[Y^{\beta}(s)Y^{\beta}(t)\right] \sim\E\left[Y^{\beta}(s)\right]\E\left[Y^{\beta}(t-s)\right]\label{gamma-beta-joint}.
 \end{align}
(ii) The asymptotic expansion of $\beta \mathbb{E}\left[Y^{2\beta}(t)B(\beta,1+\beta;Y(s)/Y(t))\right]$  is 
 \begin{align}
\beta \mathbb{E}\left[Y^{2\beta}(t)B(\beta,1+\beta;Y(s)/Y(t))\right] \sim\E\left[Y^{\beta}(s)\right]\E\left[Y^{\beta}(t-s)\right]\label{beta-joint}.
 \end{align}
(iii) For fixed $\delta>0$, the asymptotic expansion of $~\mathbb{E}\left[Y^{2\beta}(t+\delta)B(\beta,1+\beta;Y(t)/Y(t+\delta))\right]$ is
\begin{equation}\label{beta-joint-2}
\mathbb{E}\left[Y^{2\beta}(t+\delta)B(\beta,1+\beta;Y(t)/Y(t+\delta))\right]\sim B(\beta,1+\beta) \mathbb{E}\left[Y^{2\beta}(t+\delta)\right].
\end{equation}
 
\begin{proof} (i): First note that, by Stirling's approximation,
 	\begin{equation}\label{stirlings-appx}
 	\mathbb{E}[Y^{\beta}(t)]=\frac{1}{\alpha^{\beta}}\frac{\Gamma(pt+\beta)}{\Gamma(pt)}\sim\left(\frac{pt}{\alpha}\right)^{\beta},\text{ for large }t.
 	\end{equation}
 Since the gamma subordinator $\{Y(t)\}_{t\geq0}$ has stationary and independent increments, it suffices to show that
\begin{equation*}
\lim\limits_{t\rightarrow\infty}\frac{\E[Y^{\beta}(s)Y^{\beta}(t)]}{\E\left[Y^{\beta}(s)(Y(t)-Y(s))^\beta\right]}=1.
\end{equation*}
Since  $\{Y(t)\}_{t\geq0}$ is an $a.s.$ increasing process with $Y(0)=0$,
\begin{align}
&~~~Y(t)-Y(s)\leq Y(t)~~a.s.\nonumber\\
\Rightarrow &~~~\E[Y^\beta(s)(Y(t)-Y(s))^{\beta}]\leq \E[Y^{\beta}(s)Y^\beta(t)]\nonumber\\
\Rightarrow &~~~\frac{\E[Y^{\beta}(s)Y^{\beta}(t)]}{\E\left[Y^{\beta}(s)(Y(t)-Y(s))^\beta\right]}\geq 1.\label{Lemma1oneside}
\end{align}
Now consider
\begin{align}
\frac{\E[Y^{\beta}(s)Y^{\beta}(t)]}{\E\left[Y^{\beta}(s)(Y(t)-Y(s))^\beta\right]}\nonumber&=\frac{\E\left[Y^{\beta}(s)\{Y^{\beta}(t)-(Y(t)-Y(s))^\beta\}\right]}{\E\left[Y^{\beta}(s)(Y(t)-Y(s))^\beta\right]}+1\\
&\leq \frac{\E\left[Y^{\beta}(s)\{Y^{\beta}(t)-(Y^{\beta}(t)-Y^\beta(s))\}\right]}{\E\left[Y^{\beta}(s)(Y(t)-Y(s))^\beta\right]}+1\nonumber ~~ (\text{using \eqref{appnedix-1}})\\
&=\frac{\E\left[Y^{2\beta}(s)\right]}{\E\left[Y^{\beta}(s)(Y(t)-Y(s))^\beta\right]}+1.\label{Lemma1secondside}
\end{align}
From \eqref{Lemma1oneside} and \eqref{Lemma1secondside}, we have that 
\begin{equation*}
1\leq \frac{\E[Y^{\beta}(s)Y^{\beta}(t)]}{\E\left[Y^{\beta}(s)(Y(t)-Y(s))^\beta\right]} \leq \frac{\E\left[Y^{2\beta}(s)\right]}{\E\left[Y^{\beta}(s)(Y(t)-Y(s))^\beta\right]}+1.
\end{equation*}
Taking the limit as $t$ tends to infinity in the above equation and using the fact that $\{Y(t)\}_{t\geq0}$ has stationary and independent increments, we get
\begin{align*}
& 1\leq \lim\limits_{t\rightarrow\infty}\frac{\E[Y^{\beta}(s)Y^{\beta}(t)]}{\E\left[Y^{\beta}(s)(Y(t)-Y(s))^\beta\right]} \leq 1+\lim\limits_{t\rightarrow\infty}\frac{\E\left[Y^{2\beta}(s)\right]}{\E\left[Y^{\beta}(s)\right]\E\left[Y^\beta(t-s)\right]}\\
& 1\leq \lim\limits_{t\rightarrow\infty}\frac{\E[Y^{\beta}(s)Y^{\beta}(t)]}{\E\left[Y^{\beta}(s)(Y(t)-Y(s))^\beta\right]} \leq 1, ~~~ (\text{using Theorem \ref{Theorem-fnbp-dist-properties}(i)})
\end{align*}
which proves Part (i). 

\vspace{.2cm}
\noindent (ii): To prove Part (ii), it suffices to show that (in view of Part (i))
\begin{equation*}
\lim\limits_{t\rightarrow\infty}\frac{\beta \mathbb{E}\left[Y^{2\beta}(t)B(\beta,1+\beta;Y(s)/Y(t))\right]}{\E\left[Y^{\beta}(s)Y^\beta(t)\right]}=1.
\end{equation*}

\noindent Note that 
\begin{align}
B\left(\beta,1+\beta;Y(s)/Y(t)\right)&=\int_{0}^{\frac{Y(s)}{Y(t)}}u^{\beta-1}(1-u)^{\beta}du\nonumber\\
&\leq \int_{0}^{\frac{Y(s)}{Y(t)}}u^{\beta-1}du~~(\text{since }(1-u)^\beta\leq 1)\nonumber\\
&=\frac{Y^\beta(s)}{\beta Y^\beta(t)},\nonumber
\end{align}
which leads to
\begin{equation*}
\lim\limits_{t\rightarrow\infty}\frac{\beta \mathbb{E}\left[Y^{2\beta}(t)B(\beta,1+\beta;Y(s)/Y(t))\right]}{\E\left[Y^{\beta}(s)Y^\beta(t)\right]}\leq 1.
\end{equation*}
On the other hand,
\begin{align}
B\left(\beta,1+\beta;Y(s)/Y(t)\right)&=\int_{0}^{\frac{Y(s)}{Y(t)}}u^{\beta-1}(1-u)^{\beta}du\nonumber\\
&\geq \int_{0}^{\frac{Y(s)}{Y(t)}}u^{\beta-1}(1-u^\beta)du~~(\text{using \eqref{appnedix-1}})\nonumber\\
&=\frac{1}{\beta}\left(\frac{Y^\beta(s)}{Y^\beta(t)}-\frac{Y^{2\beta}(s)}{2Y^{2\beta}(t)}\right).\nonumber
\end{align}
This leads to
\begin{align}
\lim\limits_{t\rightarrow\infty}\frac{\beta \mathbb{E}\left[Y^{2\beta}(t)B(\beta,1+\beta;Y(s)/Y(t))\right]}{\E\left[Y^{\beta}(s)Y^\beta(t)\right]}&\geq \lim\limits_{t\rightarrow\infty}\frac{\beta \mathbb{E}\left[Y^{2\beta}(t)\frac{1}{\beta}\left(\frac{Y^\beta(s)}{Y^\beta(t)}-\frac{Y^{2\beta}(s)}{2Y^{2\beta}(t)}\right)\right]}{\E\left[Y^{\beta}(s)Y^\beta(t)\right]}\nonumber\\
&=\lim\limits_{t\rightarrow\infty}\frac{\mathbb{E}\left[Y^{\beta}(t)Y^\beta(s)-Y^{2\beta}(s)/2\right]}{\E\left[Y^{\beta}(s)Y^\beta(t)\right]}\nonumber\\
&=1-\lim\limits_{t\rightarrow\infty}\frac{\mathbb{E}\left[Y^{2\beta}(s)\right]}{2\E\left[Y^{\beta}(s)\right]\E\left[Y^\beta(t-s)\right]}=1,\nonumber
\end{align}
using Part (i) and Theorem \ref{Theorem-fnbp-dist-properties}(i). This completes the proof of Part (ii).

\vspace{.2cm}
\noindent (iii): It is known that if $X\sim G(\alpha,p_1)$ and $Y\sim G(\alpha,p_2)$ are two independent gamma random variables, then $U=(X+Y)$ and $V=X/(X+Y)$ are independent $G(\alpha,p_1+p_2)$ and Beta$(p_1,p_2)$ variables. Since $\{Y(t)\}_{t\geq0}$ is a gamma subordinator, $Y(t+\delta)\stackrel{d}{=}Y(t+\delta)-Y(t)+Y(t)\stackrel{d}{=}Y^{\ast}(\delta)+Y(t)$, where $Y^{\ast}(\delta)$ and $Y(t)$ are independent and hence $\frac{Y(t)}{Y^{\ast}(\delta)+Y(t)}$ and $Y^{\ast}(\delta)+Y(t)$ are independent. In other words, $\frac{Y(t)}{Y(t+\delta)}$ and $Y(t+\delta)$ are independent. Therefore,
\begin{align*}
\lim\limits_{t\rightarrow\infty}&\frac{\mathbb{E}\left[Y^{2\beta}(t+\delta)B(\beta,1+\beta;Y(t)/Y(t+\delta))\right]}{ B(\beta,1+\beta) \mathbb{E}\left[Y^{2\beta}(t+\delta)\right]}\\&=\lim\limits_{t\rightarrow\infty}\frac{\mathbb{E}\left[Y^{2\beta}(t+\delta)\right]\mathbb{E}\left[B(\beta,1+\beta;Y(t)/Y(t+\delta))\right]}{ B(\beta,1+\beta) \mathbb{E}\left[Y^{2\beta}(t+\delta)\right]}\\&=\lim\limits_{t\rightarrow\infty}\frac{\mathbb{E}\left[B(\beta,1+\beta;Y(t)/Y(t+\delta))\right]}{ B(\beta,1+\beta)}.
\end{align*}
As $t\rightarrow\infty,~Y(t)\rightarrow\infty~a.s.,~\frac{Y(t)}{Y(t+\delta)}\rightarrow 1~a.s.$ and hence $B(\beta,1+\beta;Y(t)/Y(t+\delta))\rightarrow B(\beta,1+\beta)~a.s.$ Also, $B(\beta,1+\beta;Y(t)/Y(t+\delta))(\omega)\leq B(\beta,1+\beta)$ for all $\omega$ and $t$ and hence uniformly bounded. Therefore, 
$$\lim\limits_{t\rightarrow\infty}\mathbb{E}\left[B(\beta,1+\beta;Y(t)/Y(t+\delta))\right]= B(\beta,1+\beta),$$
which proves the result.
 \end{proof}
 \end{lemma}
  We are now ready to prove the main result of this section.
 \begin{theorem}
  The FNBP $\{Q_{\beta}(t)\}_{t\geq 0}$ has the LRD property.
  \begin{proof}\noindent Consider the last term of $\text{Cov}[Q_{\beta}(s),Q_{\beta}(t)]$ given in Theorem \ref{Theorem-fnbp-dist-properties}(iii), namely,
  \begin{equation*}
   \beta q^{2} \mathbb{E}\Big[Y^{2\beta}(t)B(\beta,1+\beta;Y(s)/Y(t))\Big].
  \end{equation*}
  Using Lemma \ref{Lemma-asym}(ii), we get for large $t$,
 \begin{align}
    q^{2}\beta \mathbb{E}[Y^{2\beta}(t)B(\beta,1+\beta;Y(s)/Y(t))
    &\sim q^{2}\E[Y^{\beta}(s)]\E[Y^{\beta}(t-s)]\label{autocovariance-last-summand-1}.
 \end{align}
 \noindent By \eqref{stirlings-appx} and \eqref{autocovariance-last-summand-1}, Theorem \ref{Theorem-fnbp-dist-properties}(iii) becomes for large $t$,
 \begin{align}
  \text{Cov}[Q_{\beta}(s),Q_{\beta}(t)]&\sim q\mathbb{E}[Y^{\beta}(s)]+d\mathbb{E}[Y^{2\beta}(s)]\nonumber\\
  &~~~~~~~-q^{2}\mathbb{E}[Y^{\beta}(s)]\left(\frac{pt}{\alpha}\right)^{\beta}+q^{2}\mathbb{E}[Y^{\beta}(s)]\left(\frac{p(t-s)}{\alpha}\right)^{\beta}\nonumber\\
  &=q\mathbb{E}[Y^{\beta}(s)]+d\mathbb{E}[Y^{2\beta}(s)]-q^{2}\mathbb{E}[Y^{\beta}(s)]\left(\left(\frac{pt}{\alpha}\right)^{\beta}-\left(\frac{pt-ps}{\alpha}\right)^{\beta}\right)\nonumber\\
  &\sim q\mathbb{E}[Y^{\beta}(s)]+d\mathbb{E}[Y^{2\beta}(s)],\label{covariance-large-t}
 \end{align}
 since $t^{\beta}-(t-s)^{\beta} \sim \beta s t^{\beta-1}$ for large $t.$
 
\noindent Similarly, from Theorem \ref{Theorem-fnbp-dist-properties}(ii) and \eqref{stirlings-appx}, we have
 \begin{align}
  \text{Var}[Q_{\beta}(t)]&\sim q\left(\frac{pt}{\alpha}\right)^{\beta}-q^{2}\left(\frac{pt}{\alpha}\right)^{2\beta}+2d\left(\frac{pt}{\alpha}\right)^{2\beta}\nonumber\\
  &=t^{2\beta}\left(q\left(\frac{p}{t\alpha}\right)^{\beta}-q^{2}\left(\frac{p}{\alpha}\right)^{2\beta}+2d\left(\frac{p}{\alpha}\right)^{2\beta}\right)\nonumber\\
  &\sim  t^{2\beta}\left(\frac{p}{\alpha}\right)^{2\beta}\left(2d-q^{2}\right)\nonumber\\&=t^{2\beta}d_{1},\label{variance-large-t}
 \end{align}
 where $d_{1}=\left(p/\alpha\right)^{2\beta}R$. 
 \noindent Thus, from \eqref{covariance-large-t} and \eqref{variance-large-t}, the correlation between $Q_{\beta}(s)$ and $Q_{\beta}(t)$ for large $t>s$, is
   \begin{align}
  \text{Corr}[Q_{\beta}(s),Q_{\beta}(t)]&\sim\frac{q\mathbb{E}[Y^{\beta}(s)]+d\mathbb{E}[Y^{2\beta}(s)]}{\sqrt{t^{2\beta}d_{1}}\sqrt{\text{Var}[Q_{\beta}(s)]}}\nonumber= t^{-\beta}\left(\frac{q\mathbb{E}[Y^{\beta}(s)]+d\mathbb{E}[Y^{2\beta}(s)]}{\sqrt{d_{1}\text{Var}[Q_{\beta}(s)]}}\right),\nonumber
 \end{align}
 \noindent which decays like the power law $t^{-\beta},~0<\beta<1$ (see Definition \ref{LRD-definition}). Hence, the FNBP exhibits the LRD property.\end{proof}
 \end{theorem}
\noindent Finally, we show that the FNBN $\{Q_{\beta}^{\delta}(t)\}_{t\geq0}$ has the SRD property.
 
\begin{theorem}
 The FNBN $\{Q_{\beta}^{\delta}(t)\}_{t\geq0}$ exhibits the SRD property. 
 \end{theorem}
 
\begin{proof}
Let $s,\delta\geq0$ be fixed and $s+\delta\leq t$. By Theorem \ref{Theorem-fnbp-dist-properties}(i), we have for large $t$,
 \begin{align}
   \mathbb{E}[Q_{\beta}^{\delta}(t)] &=q(\E[Y^{\beta}(t+\delta)]-\E[Y^{\beta}(t)]) \nonumber \\
	     &\sim q\left(\frac{pt}{\alpha}\right)^{\beta}\left[\left(1+\delta/t\right)^{\beta}-1\right] \label{autocovariance-increments-1}\\
&		\sim q\beta \delta \left(\frac{p}{\alpha}\right)^{\beta} t^{\beta-1}.\label{autocovariance-increments-2} 
 \end{align}
Now using Theorem \ref{Theorem-fnbp-dist-properties}(iii), we get
 \begin{align}
   \mathbb{E}[Q_{\beta}^{\delta}(s)Q_{\beta}^{\delta}(t)]&=\mathbb{E}[Q_{\beta}(s+\delta)Q_{\beta}(t+\delta)]-\mathbb{E}[Q_{\beta}(s+\delta)Q_{\beta}(t)]-\mathbb{E}[Q_{\beta}(s)Q_{\beta}(t+\delta)]\nonumber \\
     &~~~~~~~~~~+\mathbb{E}[Q_{\beta}(s)Q_{\beta}(t)]\nonumber\\
   &=\beta q^{2}\Big( \mathbb{E}[Y^{2\beta}(t+\delta)B(\beta,1+\beta;Y(s+\delta)/Y(t+\delta))]\nonumber\\
   &~~~~~~~~~~~- \mathbb{E}[Y^{2\beta}(t)B(\beta,1+\beta;Y(s+\delta)/Y(t))]\nonumber\\
     &~~~~~~~~~~~- \mathbb{E}[Y^{2\beta}(t+\delta)B(\beta,1+\beta;Y(s)/Y(t+\delta))]\nonumber\\
       &~~~~~~~~~~~+ \mathbb{E}[Y^{2\beta}(t)B(\beta,1+\beta;Y(s)/Y(t))]\bigg)\nonumber\\
     &\sim q^{2}\big(\mathbb{E}[Y^{\beta}(s+\delta)]\mathbb{E}[Y^{\beta}(t-s)]-\mathbb{E}[Y^{\beta}(s+\delta)]\mathbb{E}[Y^{\beta}(t-s-\delta)]\nonumber\\
   &~~~~~~-\mathbb{E}[Y^{\beta}(s)]\mathbb{E}[Y^{\beta}(t-s+\delta)]+\mathbb{E}[Y^{\beta}(s)]\mathbb{E}[Y^{\beta}(t-s)]\big)~(\text{using }\eqref{beta-joint})\nonumber\\
   &\sim q^{2}\left(\frac{pt}{\alpha}\right)^{\beta}\bigg[\mathbb{E}[Y^{\beta}(s+\delta)]\left(1-\frac{s}{t}\right)^{\beta}-\mathbb{E}[Y^{\beta}(s+\delta)]\left(1-\frac{s+\delta}{t}\right)^{\beta}\nonumber\\
   &~~~~~~~~~~-\mathbb{E}[Y^{\beta}(s)]\left(1-\frac{s-\delta}{t}\right)^{\beta}+\mathbb{E}[Y^{\beta}(s)]\left(1-\frac{s}{t}\right)^{\beta}\bigg],\label{autocovariance-increments}
   \end{align}
 using \eqref{stirlings-appx}. From \eqref{autocovariance-increments-1}, we have
 \begin{equation*}
 \E[Q_{\beta}^{\delta}(s)]\E[Q_{\beta}^{\delta}(t)]\sim q^{2}\left(\frac{pt}{\alpha}\right)^{\beta}\left(\E[Y^{\beta}(s+\delta)]-\E[Y^{\beta}(s)]\right)\left(\left(1+\frac{\delta}{t}\right)^{\beta}-1\right).
 \end{equation*}
  Using $(1\pm s/t)^{\beta}\sim 1\pm\beta s/t+\beta(\beta-1)s^{2}/2t^{2},$ for large $t$, in \eqref{autocovariance-increments} and after some routine calculations, we get
     \begin{align}
     \text{Cov}[Q_{\beta}^{\delta}(s),Q_{\beta}^{\delta}(t)]&\sim q^{2}\left(\frac{pt}{\alpha}\right)^{\beta}\frac{\beta(\beta-1)}{2t^{2}}\Big[(s^{2}-(s+\delta)^{2}-\delta^{2})\E[Y^{\beta}(s+\delta)]\nonumber\\
     &~~~~~~~~~+(s^{2}+\delta^{2}-(s-\delta)^{2})\E[Y^{\beta}(s)]\Big]\nonumber\\
     &=t^{\beta-2}q^{2}\delta\left(\frac{p}{\alpha}\right)^{\beta}\beta(1-\beta)\Big((s+\delta)\mathbb{E}[Y^{\beta}(s+\delta)]-s\mathbb{E}[Y^{\beta}(s)]\Big).\label{fnbp-increments-cov}
     \end{align}
 Using $\E[Q^{2}_{\beta}(t)]=q\E[Y^{\beta}(t)]+2d\E[Y^{2\beta}(t)]$ (see Theorem \ref{Theorem-fnbp-dist-properties}(ii)), we get
   \begin{align}
   \E[(Q_{\beta}^{\delta}(t))^{2}]&=\E[Q_{\beta}^{2}(t+\delta)]+\E[Q_{\beta}^{2}(t)]-2\E[Q_{\beta}(t+\delta)Q_{\beta}(t)]\nonumber\\
   &=q\left(\E[Y^{\beta}(t+\delta)]-\E[Y^{\beta}(t)]\right)+2d\left(\E[Y^{2\beta}(t)]+\E[Y^{2\beta}(t+\delta)]\right)\nonumber\\
   &~~~~~-2d\E[Y^{2\beta}(t)]-2\beta q^{2}\E[Y^{2\beta}(t+\delta)B(\beta,1+\beta;Y(t)/Y(t+\delta))]\nonumber \\ 
   &\sim q\left(\E[Y^{\beta}(t+\delta)]-\E[Y^{\beta}(t)]\right)+2d\left(\E[Y^{2\beta}(t)]+\E[Y^{2\beta}(t+\delta)]\right)\nonumber\\
     &~~~~~-2d\big(\E[Y^{2\beta}(t)]+\E[Y^{2\beta}(t+\delta)]\big)~~(\text{using } \eqref{beta-joint-2})\nonumber \\
			&=q\left(\E[Y^{\beta}(t+\delta)]-\E[Y^{\beta}(t)]\right)\nonumber\\
       &=\E[Q^\delta_{\beta}(t)].  \nonumber 
 \end{align}
 From \eqref{autocovariance-increments-2}, we have
 \begin{align}
    \text{Var}[Q_{\beta}^{\delta}(t)]&\sim t^{\beta-1}\beta\delta q\left(\frac{p}{\alpha}\right)^{\beta}-\left(\frac{\beta\delta q p^{\beta}}{\alpha^{\beta}}\right)^{2}t^{2(\beta-1)}\nonumber\\
    &\sim t^{\beta-1}\beta \delta q\left(\frac{p}{\alpha}\right)^{\beta}.\label{fnbp-increments-var}
   \shortintertext{Thus, using \eqref{fnbp-increments-cov} and \eqref{fnbp-increments-var}, we have for fixed $s$ and large $t$,}
   \text{Corr}[Q_{\beta}^{\delta}(s),Q_{\beta}^{\delta}(t)]&\sim t^{-(3-\beta)/2}\left(\frac{q^{2}\beta(1-\beta)\delta\left(\frac{p}{\alpha}\right)^{\beta}\left((s+\delta)\mathbb{E}[Y^{\beta}(s+\delta)]-s\mathbb{E}[Y^{\beta}(s)]\right)}{\sqrt{\text{Var}[Q^{\delta}_{\beta}(s)]\beta \delta q\left(\frac{p}{\alpha}\right)^{\beta}}}\right).\nonumber
   \end{align}
   Since $(3-\beta)/2\in(1,1.5)$, for $0< \beta <1$, the increments of the FNBP possess the SRD property.
   \end{proof}
 
\begin{remark}
 Since the FPP is a non-stationary process, the FNBP is also a non-stationary process. Also, as seen earlier, the FNBP has the long-range dependence property, its increments are correlated and exhibit the short-range dependence property. Such stochastic models have wide applicability in many different areas, such as economics, finance, physics and engineering sciences.
 \end{remark}

\acks 
\noindent The authors thank the referee for some critical and encouraging comments, especially regarding the Definition 4. A part of this work was done when the second author was visiting the Department of Statistics and Probability, Michigan State
University, during Summer-2015. The authors thank Prof. Hira L. Koul for some helpful comments.

%
%
%
%

\end{document}